\documentclass[12pt]{article}

\usepackage{amsfonts,amsmath,amsthm,amssymb}
\usepackage{graphicx,bm,color}
\usepackage[T1]{fontenc}
\usepackage[utf8]{inputenc}
\usepackage[vcentering,dvips]{geometry}
\newtheorem{thm}{Theorem}

\newtheorem{remark}{Remark}
\theoremstyle{remark}

\theoremstyle{definition}

\newcommand{\R}{\mathbb{R}}

\newcommand{\N}{\mathbb{N}}

\newcommand{\E}{\mathbb{E}}

\newcommand{\eps}{{\varepsilon}}

\newcommand{\rd}{\,{\rm d}}

\newcommand{\bsx}{{\boldsymbol x}}

\newcommand{\bst}{{\boldsymbol t}}
\newcommand{\bszero}{{\boldsymbol 0}}

\newcommand{\cP}{\mathcal{P}}

\newcommand{\cA}{\mathcal{A}}

\newcommand{\nav}{{{\rm n\text{-}av}}}
\newcommand{\av}{{{\rm av}}}

\allowdisplaybreaks

\title{Generalized Discrepancy of Random Points}  
\author{Erich Novak and Friedrich Pillichshammer}
\date{\today}

\begin{document}
\maketitle

\begin{abstract}
We study the \(L_p\)-discrepancy of random point sets in high dimensions, 
with emphasis on small values of \(p\). Although the classical 
\(L_p\)-discrepancy suffers from the curse of dimensionality for all 
\(p \in (1,\infty)\), the gap between known upper and lower bounds remains 
substantial, in particular for small $p \ge 1$. 
To clarify this picture, we review the existing results for 
i.i.d.\ uniformly distributed points and derive new upper bounds for 
\emph{generalized} \(L_p\)-discrepancies, obtained by allowing non-uniform 
sampling densities and corresponding non-negative quadrature weights.

Using the probabilistic method, we show that random points drawn from 
optimally chosen product densities lead to significantly improved upper bounds. 
For \(p=2\) these bounds are explicit and optimal; for general 
\(p \in [1,\infty)\) we obtain sharp asymptotic estimates. The improvement can 
be interpreted as a form of importance sampling for the underlying Sobolev 
space \(F_{d,q}\).

Our results also reveal that, even with optimal densities, the curse of 
dimensionality persists for random points when \(p\ge1\), and it becomes most 
pronounced for small \(p\). This suggests that the curse should also hold for 
the classical \(L_1\)-discrepancy for deterministic point sets.
\end{abstract} 

\centerline{\begin{minipage}[hc]{130mm}{
}
\end{minipage}}

\section{Introduction}
\label{sec:intro}

The $L_p$-discrepancy of point sets in $[0,1]^d$ is one of the central
quantities in quasi-Monte Carlo (QMC) theory. It measures the worst-case
integration error for functions belonging to the Sobolev space underlying the
Koksma-Hlawka inequality. In this paper we study the generalized
$L_p$-discrepancy associated with the multivariate Sobolev space $F_{d,q}$,
where $p$ and $q$ are Hölder conjugates. This space consists of functions that
are once differentiable in each coordinate, whose mixed derivative lies in
$L_q([0,1]^d)$, and which satisfy the boundary condition $f(\bsx)=0$ whenever at
least one coordinate of $\bsx \in [0,1]^d$ equals~$1$. These spaces appear naturally in
multivariate integration and tractability theory.

A classical result (see, e.g., \cite{NW10}) states that for QMC algorithms
with equal weights, the worst-case error in $F_{d,q}$ equals the classical
$L_p$-discrepancy of the underlying point set. For general linear algorithms
with non-negative (but not necessarily equal) weights, one obtains the more
flexible notion of a \emph{generalized $L_p$-discrepancy}. This extended
framework is particularly useful in high dimensions: it allows one to
incorporate non-uniform sampling densities and to interpret the analysis in
terms of importance sampling for the Sobolev space~$F_{d,q}$.

\medskip

It is known that the classical $L_p$-discrepancy suffers from the curse of
dimensionality for all $p\in(1,\infty)$ when restricted to non-negative
weights \cite{NP23,NP24}. More precisely, the number of required sample points
grows exponentially in~$d$ when aiming at a fixed relative error
reduction. Nevertheless, the best available upper bounds
are still far from matching
these lower bounds. The gap is especially large for small values of $p$, and
the case $p=1$ remains the least understood. Existing explicit bounds for the
average $L_p$-discrepancy of uniformly distributed random points are known
only for even integers $p\ge2$, and asymptotic results for general
$p\in[1,\infty)$ are due to Steinerberger \cite{St2010} and Hinrichs--Weyhausen
\cite{HW12}. Much less is known about how one can improve these bounds by
choosing the sampling distribution in a non-uniform way.

\medskip

The present paper investigates the effect of non-uniform random sampling.
Instead of choosing points independently and uniformly in $[0,1]^d$, we draw
them independently according to a product density $\varrho_d=\varrho^{\otimes
d}$ and evaluate the integrand with weights proportional to $1/\varrho_d$.
This leads to the generalized discrepancy $L_{p,N}(\mathcal{P},\mathcal{A})$,
whose expectation with respect to the density~$\varrho_d$ we analyze in
detail. The probabilistic method then implies the existence of deterministic
point sets with discrepancy no larger than this expectation.

\medskip

Our first goal is to clarify to what extent a suitable change of measure can
improve the probabilistic upper bounds in high dimensions. The answer turns
out to depend strongly on~$p$. For $p=2$ we completely solve the underlying
variational problem and recover the optimal density already found implicitly
in earlier work of Plaskota, Wasilkowski, and Zhao \cite{Petal}. This yields
explicit optimal constants and improves the classical bound by a factor of
$(4/3)^{d/2}$ instead of $(3/2)^{d/2}$. For general $p\in[1,\infty)$ we show
that the optimal one-dimensional density is uniquely determined by a
nonlinear equation and leads to a base of the form
\[
  \alpha_p = \Bigl(\tfrac{p+2}{p+1}\Bigr)^{1/2}
\]
in the exponential factor of the $d$-dependence. This improves the
corresponding base $((2p+2)/(p+2))^{1/p}$ from the classical uniform-sampling bound.
The effect is most pronounced for small~$p$ and diminishes as $p\to\infty$.

\medskip

A second goal is to clarify how random information compares with optimal
deterministic constructions. Our upper bounds show that random points drawn
from optimally chosen product measures can significantly outperform all
currently known deterministic point sets with respect to generalized
discrepancy. At the same time, even for the optimal density, the curse of
dimensionality persists for all $p\ge1$. Indeed, the exponential growth rate
in~$d$ becomes \emph{more} severe as $p\to1$. This suggests that the curse
should also hold for the classical $L_1$-discrepancy with deterministic
points---a problem that is still open.

\medskip

The paper is organized as follows. In Section~\ref{sec:p2} we study the case
$p=2$ in detail, solve the associated variational problem explicitly, and
derive the optimal density. Section~\ref{sec:generalp} extends the analysis to
general $p\in[1,\infty)$, leading to an implicit characterization of the
optimal density and sharp asymptotic bounds for the average discrepancy. In
Section~\ref{stability} we discuss the stability of the resulting
quadrature formulas in the norm of~$F_{d,q}$ and present open problems.
But first we give some further background.


\section{Background}\label{sec:back} 

We study numerical integration for the function space that is the canonical 
space underlying the Koksma-Hlawka inequality as presented in \cite[Chapter~9, Section~9.8]{NW10}. 
The goal is to investigate how probabilistic arguments can improve upper 
bounds for generalized discrepancy measures and the corresponding integration problem.

\paragraph{The function space setting.}
Throughout this paper let $p,q \ge 1$ be H\"older conjugates, i.e.,  $1/p + 1/q = 1$. 
For $d = 1$, let $W_q^1([0,1])$ denote the space of absolutely continuous 
functions whose first derivatives belong to $L_q([0,1])$. 
For $d>1$ we consider the $d$-fold tensor product space
\[
W_q^{\boldsymbol{1}} := W_q^{(1,1,\ldots,1)}([0,1]^d),
\]
which is the Sobolev space of functions on $[0,1]^d$ that 
are once differentiable in each variable and whose mixed derivative 
$\partial^d f / \partial \bsx$ has finite $L_q$-norm, where $\partial \bsx 
= \partial x_1 \partial x_2 \ldots \partial x_d$. 

Now consider the subspace of functions satisfying the boundary conditions\footnote{%
In analysis and stochastics the boundary condition $f(0)=0$ is more common. 
Of course it does not really matter, since one can always change a variable 
$t\in(0,1)$ to $1-t$ to switch between the two cases; see also \cite[p.~24]{NW10}.}  
$f(\bsx)=0$ if at least one component of $\bsx=(x_1,\ldots,x_d)$ equals $1$, 
and equip this subspace with the norm
\[
\|f\|_{d,q} := \Biggl(\int_{[0,1]^d} 
  \Bigl| \frac{\partial^d}{\partial \bsx}f(\bsx)\Bigr|^q \rd \bsx \Biggr)^{1/q}
  \quad \text{for } q\in[1,\infty),
\]
and 
\[
\|f\|_{d,\infty} := 
  \sup_{\bsx \in [0,1]^d}
  \Bigl|\frac{\partial^d}{\partial \bsx}f(\bsx)\Bigr|
  \quad \text{for } q=\infty.
\]
That is,
\[
F_{d,q} := \Bigl\{
  f \in W_q^{\boldsymbol{1}} :
  f(\bsx)=0 \text{ if } x_j=1 \text{ for some } j\in[d],\ 
  \|f\|_{d,q}<\infty
  \Bigr\},
\]
where $[d]:=\{1,2,\ldots,d\}$.

\paragraph{Integration and linear algorithms.}
Now consider multivariate integration
\begin{equation}\label{int:probl}
I_d(f) := \int_{[0,1]^d} f(\bsx)\,\rd \bsx
\quad\text{for } f\in F_{d,q}.
\end{equation}
We approximate $I_d(f)$ by linear algorithms of the form
\begin{equation}\label{def:linAlg}
A_{d,N}(f) = \sum_{k=1}^N a_k f(\bst_k),
\end{equation}
where $\bst_1,\ldots,\bst_N \in [0,1)^d$ are sample points and 
$a_1,\ldots,a_N$ are real integration weights. 
If $a_1=\dots=a_N=1/N$, then \eqref{def:linAlg} is a quasi-Monte~Carlo (QMC) 
algorithm, denoted by $A_{d,N}^{\ast}$. 

In this paper we derive new upper bounds for \emph{non-negative} quadrature formulas, i.e.\ those with 
$a_k\ge0$ for all $k\in\{1,\dots,N\}$. 
To emphasize that only non-negative weights are used we write $A_{d,N}^+$. 
Such non-negative integration weights are often favored in practice 
because of their monotonicity and numerical stability (see Section~\ref{stability}). 

For lower bounds, the restriction to positive quadrature formulas makes the results weaker, 
but here we focus entirely on \emph{upper} bounds.

\paragraph{Worst-case error and generalized discrepancy.}
Define the worst-case error of an algorithm \eqref{def:linAlg} by
\begin{equation}\label{eq:wce}
e(F_{d,q},A_{d,N})
   = \sup_{\substack{f \in F_{d,q}\\ \|f\|_{d,q}\le1}}
     \bigl| I_d(f)-A_{d,N}(f)\bigr|.
\end{equation}
For a QMC algorithm $A_{d,N}^{\ast}$ it is well-known (see, e.g., \cite[Section~9.8.1]{NW10}) that
\[
e(F_{d,q},A_{d,N}^{\ast}) = L_{p,N}^{\ast}(\cP),
\]
where $L_{p,N}^{\ast}(\cP)$ is the classical $L_p$-discrepancy of the 
point set $\cP = \{\bst_1,\ldots,\bst_N\}$.

For general linear algorithms \eqref{def:linAlg} the worst-case error equals the 
\emph{generalized $L_p$-discrepancy}
\begin{equation}\label{dual}
e(F_{d,q},A_{d,N}) = L_{p,N}(\cP,\cA),
\end{equation}
where $\cA=\{a_1,\ldots,a_N\}$ are the integration weights corresponding to $\cP$. 
This equality follows from the Koksma-Hlawka inequality; see \cite[Section~9.8.1]{NW10}. 

For points $\cP=\{\bst_k\}$ and weights $\cA=\{a_k\}$, the discrepancy function is
\[
\Delta_{\cP,\cA}(\bsx)
   = \sum_{k=1}^N a_k\,\mathbf{1}_{[\boldsymbol{0},\bsx)}(\bst_k) 
     - x_1x_2\cdots x_d,
\]
for $\bsx=(x_1,\ldots,x_d)\in[0,1]^d$, and the generalized $L_p$-discrepancy is
\[
L_{p,N}(\cP,\cA)
  = \Biggl(\int_{[0,1]^d}
        |\Delta_{\cP,\cA}(\bsx)|^p \rd\bsx
    \Biggr)^{1/p}\qquad \mbox{for $p\in[1,\infty)$,}
\]
with the usual modification for $p=\infty$. If $a_1=\dots=a_N=1/N$, we recover the classical definition $L_{p,N}^{\ast}(\cP)$.

We define the $N$th minimal worst-case error
\[
e_q(N,d)
  := \min_{A_{d,N}} e(F_{d,q},A_{d,N}),
\]
where the minimum ranges over all algorithms of the form~\eqref{def:linAlg}. 
The initial error (i.e., $N=0$) equals
\[
e_q(0,d)
   = \sup_{\|f\|_{d,q}\le1} |I_d(f)|
   = \begin{cases}
      (p+1)^{-d/p} & \mbox{if $q\in(1,\infty]$,}\\[0.3em]
      1 & \mbox{if $q=1$.}
     \end{cases}
\]

Using the duality \eqref{dual}, one obtains
\begin{equation}\label{dual:erdisc}
e_q(N,d)
   = {\rm disc}_p(N,d)
   := \min_{\cP,\cA} L_{p,N}(\cP,\cA),
\end{equation}
and we refer to ${\rm disc}_p(N,d)$ as the \emph{$N$th minimal $L_p$-discrepancy}. 

\paragraph{Information complexity.}
We define the information complexity as the minimal number of function 
evaluations needed to reduce the initial error by a factor $\varepsilon$, where $\varepsilon >0$:  
\[
N_q^{\rm int}(\varepsilon,d)
   := \min\{ N\in\N : e_q(N,d) \le \varepsilon\,e_q(0,d)\}.
\]
By duality,
\[
N_q^{\rm int}(\varepsilon,d)
  = N_p^{\rm disc}(\varepsilon,d),
\quad
N_p^{\rm disc}(\varepsilon,d)
  := \min\{N: {\rm disc}_p(N,d)\le \varepsilon\,{\rm disc}_p(0,d)\}.
\]

In the QMC case ($a_k=1/N$), we add a superscript ${\ast}$, 
writing $e_q^{\ast}(N,d)$, ${\rm disc}_p^{\ast}$, $N_q^{{\rm int},{\ast}}$, etc.  
When we restrict to non-negative weights ($a_k\ge0$), we write the superscript $+$.  
Obviously,
\[
e_q(N,d)\le e_q^+(N,d)\le e_q^{\ast}(N,d),
\qquad
N_q^{\rm int}(\varepsilon,d)
  \le N_q^{{\rm int},+}(\varepsilon,d)
  \le N_q^{{\rm int},{\ast}}(\varepsilon,d),
\]
and analogous relations hold for the corresponding discrepancies.\\

An important question now is, how quickly the information complexity $N_q^{{\rm int}}(\varepsilon,d)$ increases, 
when $d \rightarrow \infty$ and/or $\varepsilon \rightarrow 0$ (i.e., when $d+\varepsilon^{-1} \rightarrow \infty$). 
Such questions are typically studied in the field ``Information Based Complexity'' (see \cite{NW08,NW10,NW12}). 
If $N_q^{{\rm int}}(\varepsilon,d)$ grows at least exponentially fast in $d$ for a fixed $\varepsilon >0$, 
more precisely, if there exists a real $C>1$ such that $N_q^{{\rm int}}(\varepsilon,d) \ge C^d$ 
for infinitely many $d \in \N$, then the integration  problem is said to suffer from the curse of dimensionality.

If, on the other hand, $N_q^{{\rm int}}(\varepsilon,d)$ grows at most sub-exponentially fast 
in $d$ and $\varepsilon^{-1}$ for $d+\varepsilon^{-1} \rightarrow \infty$, then the integration 
problem is said to be tractable and there are various notions of tractability in order 
to classify the grow rate of $N_q^{{\rm int}}(\varepsilon,d)$ more accurately\footnote{For example, the integration problem  is said to be polynomially tractable, if there exist numbers $C,\sigma>0$ and $\tau \ge 0$ such that $N_q^{{\rm int}}(\varepsilon,d)\le C d^{\tau} \varepsilon^{- \sigma}$ for all $\varepsilon \in (0,1)$ and $d \in \N$. If this holds even for $\tau=0$, the one speaks about strong polynomial tractability.}.

\medskip

In our recent papers \cite{NP23,NP24,NP24a} we proved \emph{lower bounds}, showing that the 
curse holds (in terms of discrepancy) for all \(p\in(1,\infty)\) for non-negative quadrature formulas, 
whereas it disappears for \(p=\infty\) (\cite{HNWW01}); the case \(p=1\) remains open. 
Now we derive \emph{upper bounds} using the probabilistic method, i.e.\ we analyze the 
behavior of random points and compute or estimate expected discrepancies.

\paragraph{The probabilistic method in discrepancy theory.}
The probabilistic method was first used in a deep result of Hlawka~\cite{Hl47} (1947) 
to prove the existence of a lattice sphere packing with large density, 
building on Siegel’s mean value theorem~\cite{Sie45}. 
Almost simultaneously, Erd\H{o}s~\cite{E47} employed the method for 
existence theorems in combinatorics and graph theory.

Regarding the $L_{\infty}^{\ast}$- (star) discrepancy, Heinrich et~al.~\cite{HNWW01} proved
\[
N_{\infty}^{{\rm disc},\ast}(d,\varepsilon) \le C\, d\, \varepsilon^{-2},
\]
where $C>0$ is an absolute constant.  
Equivalently, for every $d$ and $N$ there exists an $N$-element point set $\cP\subseteq [0,1)^d$ such that
\[
L_{\infty,N}^{\ast}(\cP) \le \sqrt{C}\,\sqrt{\tfrac{d}{N}}.
\]
For this and related results on average $L_p^{\ast}$-discrepancies (for even~$p$), 
they applied the probabilistic method to i.i.d.\ uniformly distributed points in $[0,1]^d$. 
Let $\cP_N$ denote an i.i.d.\ sample of $N$ uniform random points in $[0,1]^d$. 
Then, with high probability,
\[
L_{\infty,N}^{\ast}(\cP_N)\le C\sqrt{\tfrac{d}{N}},
\]
see Aistleitner and Hofer~\cite{AH14}. Doerr~\cite{Doe14} showed that this order cannot be improved 
substantially: there exist absolute constants $C_1,C_2>0$ with
\[
C_1\sqrt{\tfrac{d}{N}} \le \E[L_{\infty,N}^{\ast}(\cP_N)] \le C_2\sqrt{\tfrac{d}{N}}.
\]

In weighted function spaces, where the relevance of 
coordinates decays, QMC methods can achieve dimension-independent bounds.
If, in the sense of Sloan and Wo\'{z}niakowski~\cite{SW98}, coordinate weights $\boldsymbol{\gamma}=(\gamma_j)$ control variable importance, then randomized lattice rules 
and digital nets satisfy, for every $\delta\in(0,1)$, the existence of a positive $C(\boldsymbol{\gamma},\delta)$ such that
\[
\E[L_{\infty,N,\boldsymbol{\gamma}}^{\ast}(\cP_N)] \le C(\boldsymbol{\gamma},\delta)\,N^{-1+\delta},
\]
leading to strong polynomial tractability (see, e.g., \cite{DickKuoSloan2013,DKP,DickPillichshammer2010,NP25c,SW98}). 

Randomized shifts and scrambles are analyzed probabilistically, 
and expectation bounds again imply the existence of good deterministic point sets.  
These results ensure polynomial or strong polynomial tractability (see \cite{NW08,NW10}).  
For some, though not all, of these upper bounds there exist constructive versions.

\medskip

In this paper we concentrate on the integration problem related to the classical $L_p^{\ast}$-discrepancy for
finite $p$, i.e., $p \in [1,\infty)$. 
We use random i.i.d.\ points \(\{\bst_1,\ldots,\bst_N\}\) in $[0,1]^d$. 
However, uniform sampling is not appropriate for the underlying Sobolev space \(F_{d,q}\), 
because functions in \(F_{d,q}\) satisfy \(f(\bsx)=0\) whenever any coordinate \(x_j=1\); 
hence their values are more significant for small coordinates (where \(f\) can be large) 
than for large ones (where \(f\) must be small). Consequently, regions near the origin contribute most to the integration error.

This observation motivates a change of measure, analogous to importance sampling in Monte Carlo methods~\cite{MNR2012}.
Even though we are ultimately interested in deterministic upper bounds for quadrature 
formulas (or, equivalently, for the \(L_p^{\ast}\)-discrepancy), we employ a different 
averaging strategy in the probabilistic method. There is, however, one caveat for the 
discrepancy setting: we must replace the equal QMC weights \(1/N\) by new non-negative 
weights \(a_k = 1/(N\varrho_d(\bst_k))\), where \(\varrho_d\) is a suitably optimized 
density. Hence our new upper bounds apply to the \emph{generalized} \(L_p\)-discrepancy. 

In more detail, we use random i.i.d.\ points \(\cP_N=\{\bst_1,\ldots,\bst_N\}\) in $[0,1]^d$ 
drawn according to a density function~\(\varrho_d\) on \([0,1]^d\), and consider algorithms of the form
\begin{equation}\label{alg:rho}
A_{d,N}(f)
   = \frac{1}{N}\sum_{k=1}^N\frac{1}{\varrho_d(\bst_k)}f(\bst_k).
\end{equation}
We write $\cA_{\varrho_d}$ for the set of these weights. We then focus on the $p$-average
\[
\av_p(N,d,\varrho_d):=\E\bigl[(L_{p,N}(\cP_N,\cA_{\varrho_d}))^p\bigr]^{1/p},
\]
and its normalized version
\[
\nav_p(N,d,\varrho_d)
   := (p+1)^{d/p}\,\E\bigl[(L_{p,N}(\cP_N,\cA_{\varrho_d}))^p\bigr]^{1/p},
\]
where in both cases the expectation is with respect to the density $\varrho_d$. 
The classical QMC-case corresponds to $\varrho_d= 1$. We then write simply $\av_p^{\ast}(N,d)$ and $\nav_p^{\ast}(N,d)$, respectively.

Note that an upper bound of the form
\[
\nav_p(N,d,\varrho_d)\le C_p\,\alpha_p^d\,N^{-1/2},
\]
for some non-negative $C_p,\alpha_p$, implies
\begin{equation}\label{bd:comp:av}
N_p^{{\rm disc}}(\varepsilon,d)
  \le  C_p^2\,\alpha_p^{2d}\,\varepsilon^{-2} .
\end{equation}

\paragraph{Main questions.}
We aim to contribute to two fundamental questions:
\begin{enumerate}
\item How can we establish sharper \emph{upper bounds} for the integration problem related to the classical \(L_p^{\ast}\)-discrepancy?
\item How good is random information compared with deterministic constructions?
For a moderate number of points $\bst_1,\ldots,\bst_N\in[0,1]^d$,
it turns out that random information (random choice of the points)
outperform all presently known deterministic constructions.
Existing results usually consider points that are i.i.d.\ and uniformly distributed. 
We demonstrate that substantially better results arise when the random points follow 
product measures whose densities are concentrated toward small values of \(x\).
See \cite{HKNPU} and \cite{SoUl} for recent surveys on i.i.d.\ sampling. 
\end{enumerate}

\paragraph{Asymptotic analysis.} Throughout, for functions $f,g: \N \rightarrow \R$ we 
write $f(N)\sim g(N)$ for $N \rightarrow \infty$, if $\lim_{N \rightarrow \infty} f(N)/g(N)=1$ 
(the usual asymptotic equivalence in the sense of de Bruijn~\cite{dB}) 
and we write $f(N) \lesssim g(N)$ for $N \rightarrow \infty$, if  $\lim_{N \rightarrow \infty} f(N)/g(N) \le 1$. 

\section{Change of measure for $p=2$}\label{sec:p2}

Consider the integration problem \eqref{int:probl} for functions from the Sobolev-Hilbert space \(F_{d,2}\). The reproducing kernel of this space is the tensor product of \(K_1(x,y)=1-\max(x,y)\), and the representer of the integral is 
\(h_d(\bsx)=\prod_{j=1}^d \tfrac12(1-x_j^2)\) 
with \(\|h_d\|_{d,2}=3^{-d/2}\).

We use random i.i.d.\ points \(\bst_1,\ldots,\bst_N\in[0,1]^d\)
drawn according to a density function~\(\varrho_d\) on \([0,1]^d\),
and consider algorithms of the form
\[
A_{d,N}(f)
   = \frac{1}{N}\sum_{k=1}^N\frac{1}{\varrho_d(\bst_k)}f(\bst_k).
\]
Using \cite[Eq.~(9.46)]{NW10} it is straightforward to compute the squared worst-case error:
\[
e(F_{d,2},A_{d,N})^2
   = \frac{1}{3^d}
     -\frac{2}{N}\sum_{k=1}^Nh_d(\bst_k)\frac{1}{\varrho_d(\bst_k)}
     +\frac{1}{N^2}\sum_{k,\ell=1}^N
        K_d(\bst_k,\bst_\ell)
        \frac{1}{\varrho_d(\bst_k)\varrho_d(\bst_\ell)}.
\]
We assume that
\begin{equation}\label{C(K)}
C(K_d,\varrho_d)
   :=\int_{[0,1]^d}K_d(\bst,\bst)\frac{1}{\varrho_d(\bst)}\,\rd\bst
   <\infty.
\end{equation}

Taking the expectation over the random points yields
\[
\E[e(F_{d,2},A_{d,N})^2]
   = \frac{C(K_d,\varrho_d)-3^{-d}}{N}.
\]
Hence, by the mean-value argument,
for certain deterministic points \(\bst_1,\ldots,\bst_N\) we have
\begin{equation}\label{upper}
e(F_{d,2},A_{d,N})^2
   \le \frac{C(K_d,\varrho_d)-3^{-d}}{N}.
\end{equation}
If we take uniformly distributed points, i.e., $\varrho_d=\lambda_d$, 
then we recover the well-known 
equation
\begin{equation}\label{eq3}
\nav_2^{\ast}(N,d) = 3^{d/2}\sqrt{\tfrac{2^{-d}-3^{-d}}{N}} \, \lesssim \,  
\Bigl(\tfrac{3}{2}\Bigr)^{d/2}N^{-1/2},
\end{equation}
for all \(N\) and \(d\).

Comparing \eqref{C(K)} and~\eqref{upper}, 
it is evident that we want to minimize \(C(K_d,\varrho_d)\)
over all densities~\(\varrho_d\).
This optimization problem has a unique solution, 
which follows from convexity via the Euler-Lagrange equations.
For \(d=1\) one needs to minimize
\begin{equation*}
C(K_1,\varrho)
   = \int_0^1(1-t)\frac{1}{\varrho(t)}\,\rd t,
\end{equation*}
and the solution is (see also Theorem~\ref{thm:opt:meas})
\begin{equation}\label{optdenp2d1}
\varrho^{\ast}(t) = \tfrac{3}{2}\sqrt{1-t}.
\end{equation}
We then obtain the minimal value \(C(K_1,\varrho^{\ast})= \tfrac{4}{9}\). The density \eqref{optdenp2d1} was already used in \cite{NW10,Petal}, here we clearly see that it is optimal.

For \(d\in \mathbb{N}\) the optimal density $\varrho_d^{\ast}$ is the $d$-fold tensor product $\varrho^{\ast,\otimes d}:=\bigotimes_{j=1}^d \varrho^{\ast}$, leading to the following result, which improves~\eqref{eq3}:

\begin{thm}
For the optimal density $\varrho_d^{\ast}$ we have  
\begin{equation}\label{eq4}
\nav_2(N,d,\varrho_d^{\ast}) =  3^{d/2}\sqrt{\tfrac{(4/9)^d-3^{-d}}{N}}
\, \lesssim \, \Bigl(\tfrac{4}{3}\Bigr)^{d/2}N^{-1/2}.
\end{equation}
We have $\varrho_d^{\ast}=\varrho^{\ast,\otimes d}$, where $\varrho^{\ast}$ is given in \eqref{optdenp2d1}.
\end{thm}

The factor \(3^{d/2}\) in \eqref{eq4} stems from normalization. 

\begin{remark}\rm
It is known that the weights \(1/(N\varrho_d(\bst_k))\) are not optimal:
for small \(N\) one should use smaller weights.
A simple modification is to introduce a constant \(c>0\) and define
\[
A_{d,N}^{(c)}(f)
   := \frac{c}{N}\sum_{k=1}^N\frac{1}{\varrho_d(\bst_k)}f(\bst_k).
\]
Computing
\(\E[e(F_{d,2},A_{d,N}^{(c)})^2]\) and optimizing over \(c\) gives
\[
c_{\ast}
   = \frac{N}{N-1+3^{d}C(K_d,\varrho_d)}.
\]
With this choice we obtain
\[
\E[e(F_{d,2},A_{d,N}^{(c_{\ast})})^2]
   =  \frac{N}{N-1+3^{d}C(K_d,\varrho_d)}\, \E[e(F_{d,2},A_{d,N})^2].
\]
Hence, for \(N\le 3^{d}C(K_d,\varrho_d)\) a significant improvement is achieved, 
whereas for very large \(N\) the benefit becomes negligible.
\end{remark}

\begin{remark}\rm
It was shown by Hinrichs~\cite{Hi2010} that for \(p=2\) the curse of dimensionality 
(for the normalized error) can be avoided in the randomized setting using 
importance sampling with a suitable density, generally unknown in closed form.
Here one can proceed similarly to the survey~\cite[pp.~159-160]{Hi}:
sampling with the representer \(h_d(\bsx)=\prod_{j=1}^d \tfrac12(1-x_j^2)\). 
Then the Monte Carlo error is bounded by the norm of integration 
(the initial error) times \(N^{-1/2}\);
thus there is \emph{no curse for the normalized error}.
The proof closely follows \cite[p.~160]{Hi}.
In fact, the normalized density proportional to \(h_d\) is optimal 
in the randomized setting for importance sampling.

Interestingly, the two optimal densities---for the deterministic worst-case 
setting and for the randomized setting---are similar but not identical.
The optimal density \eqref{optdenp2d1} minimizes the worst-case error 
for deterministic algorithms but is not optimal for randomized importance sampling 
and cannot remove the curse entirely.
\end{remark}

\section{Change of the measure for general $p$}
\label{sec:generalp} 

First we review known explicit upper bounds and their 
limitations. As a tool to prove upper bounds for the star-discrepancy (\(p=\infty\)), 
the authors of~\cite{HNWW01} computed the average \(L_p^{\ast}\)-discrepancy 
for {\it even} integers~\(p\). The exact expressions are lengthy, but after normalization 
(multiplication by \((p+1)^{d/p}\)) they lead to the bound
\begin{equation}\label{eq10}
\nav_p^{\ast}(N,d)
   \le 3^{2/3}\,2^{5/2}\,p
        \Bigl(\tfrac{2p+2}{p+2}\Bigr)^{d/p}
        N^{-1/2}.
\end{equation}
Gnewuch~\cite{Gne05} proved, that this estimate can be improved by symmetrization, 
yielding\footnote{The version stated in~\cite{HW12} and the Oberwolfach report~\cite{HN04} contains a minor error.}
\begin{equation}\label{eq11}
\nav_p^{\ast}(N,d) \le 2^{1/2}p^{1/2} \Bigl(\tfrac{2p+2}{p+2}\Bigr)^{d/p} N^{-1/2}.
\end{equation}
The upper bound~\eqref{eq11} is quite sharp, since the factor \(2^{1/2}p^{1/2}\)
exceeds the asymptotically correct constant that will be given in~\eqref{uniform} by at most a factor of \(\sqrt{2\,{\rm e}}\approx 2.33\).

We emphasize that both bounds \eqref{eq10} and \eqref{eq11} are proved only for even integers~\(p\); 
they provide information on the curse of dimensionality only for \(p\ge2\). 
Much less is known for small \(p\), and no comparable results exist for \(p<2\). 

The main part in the bounds~\eqref{eq10} and \eqref{eq11}
is the same factor \(\bigl(\tfrac{2p+2}{p+2}\bigr)^{d/p}\),
which in view of \eqref{bd:comp:av} implies that the number of required points grows as
\begin{equation}\label{order}
N \asymp \Bigl(\tfrac{2p+2}{p+2}\Bigr)^{2d/p}.
\end{equation}
Observe that the cost \(N\) increases exponentially in \(d\),
but the bound becomes more favorable for large \(p\):
for \(p=2,\,10,\text{ and }100\) we obtain
\((2p+2)/(p+2))^{2/p} = 1.5,\ 1.13,\ \text{and }1.014\), respectively.

One of the main goals of this paper is to improve the base in~\eqref{order} 
by means of a  change of measure
and to understand the behavior also for \(p<2\). 
For general finite \(p\), we obtain asymptotic relations that hold for large \(N\) only. 
The asymptotic result due to Steinerberger~\cite{St2010} and 
Hinrichs and Weyhausen~\cite{HW12}, valid for uniformly distributed points, is
\begin{equation}\label{uniform}
\nav_p^{\ast}(N,d)
\, \lesssim \, 
   \tfrac{\sqrt{2}}{\pi^{1/(2p)}}
   \Gamma\!\left(\tfrac{p+1}{2}\right)^{1/p}
   \Bigl(\tfrac{2p+2}{p+2}\Bigr)^{d/p}
   N^{-1/2}
   \quad \text{as } N\to\infty.
\end{equation}
We improve this by optimizing the change of measure, obtaining
\begin{equation}\label{new-density}
\nav_p(N,d,\varrho_d^{\ast})
   \lesssim 
   \tfrac{\sqrt{2}}{\pi^{1/(2p)}}
   \Gamma\!\left(\tfrac{p+1}{2}\right)^{1/p}
   \Bigl(\tfrac{p+2}{p+1}\Bigr)^{d/2}
   N^{-1/2}
   \quad \text{as } N\to\infty,
\end{equation}
which holds for the optimal density described in Theorem~\ref{thm:opt:meas}. 
We would like to point out that in all these relationships 
$\lesssim$ could be replaced by $\sim$ if $d$ together with $N$ tends towards infinity.
Using Stirling’s formula, it follows that
\[
\Gamma\!\left(\tfrac{p+1}{2}\right)^{1/p}
   \sim \sqrt{\tfrac{p}{2\, {\rm e}}}
   \qquad \text{as } p\to\infty.
\]

In the most interesting case \(p=1\), these results read
\begin{equation}\label{eq7}
\nav_1^{\ast}(N,d)
  \, \lesssim \,  \sqrt{\tfrac{2}{\pi}}
         \Bigl(\tfrac{4}{3}\Bigr)^d N^{-1/2}
   \quad\text{as } N\to\infty,
\end{equation}
which follows from \cite[Example~1]{HW12}. Using the optimal one-dimensional density
\begin{equation}\label{optimal1}
\varrho^{\ast}(t) = 1 + 2\cos\!\Big(\tfrac{1}{3}\arccos(2t-1) + \tfrac{4\pi}{3}\Big)\quad \mbox{for $t \in [0,1]$}
\end{equation}
from Theorem~\ref{thm:opt:meas}, \eqref{eq7} is  improved to
\begin{equation}\label{eq8}
\nav_1(N,d,\varrho_d^{\ast})
 \,  \lesssim \, \sqrt{\tfrac{2}{\pi}}\,1.225^{d}N^{-1/2}
   \quad\text{as }N\to\infty,
\end{equation}
where the value \( \sqrt{3/2} \approx 1.225 \) arises from the density~\eqref{optimal1}.
Even with this optimized density the curse of dimensionality remains rather severe 
for random points and \(p=1\); it appears more pronounced than for larger~\(p\).
Hence we conjecture that the curse also persists for \(p=1\) for optimally chosen deterministic points.

\medskip

To visualize the improvement achieved by the change of measure,  
we display in Figure~\ref{fig:alpha} the two functions $\alpha^2_{\text{old}}(p)$ and $\alpha^2_{\text{new}}(p)$, where  
\begin{equation}\label{def:alpha:old:new}
\alpha_{\text{old}}(p)
   := \Bigl(\tfrac{2p+2}{p+2}\Bigr)^{1/p},
   \qquad
\alpha_{\text{new}}(p)
   := \Bigl(\tfrac{p+2}{p+1}\Bigr)^{1/2}. 
\end{equation}
These functions appear in  \eqref{uniform} and~\eqref{new-density}, respectively, 
if we ask how $N$ depends on $\varepsilon$ for a fixed $d$. They correspond directly to 
the quantity $\alpha_p$ in \eqref{bd:comp:av}. The plot clearly shows that $\alpha^2_{\text{old}}(p)$ is strictly 
larger than $\alpha^2_{\text{new}}(p)$ for all $p$ in this range, and that 
both decrease monotonically toward~1 as \(p\to\infty\).
\begin{figure}
\centering
\includegraphics[width=0.6 \textwidth]{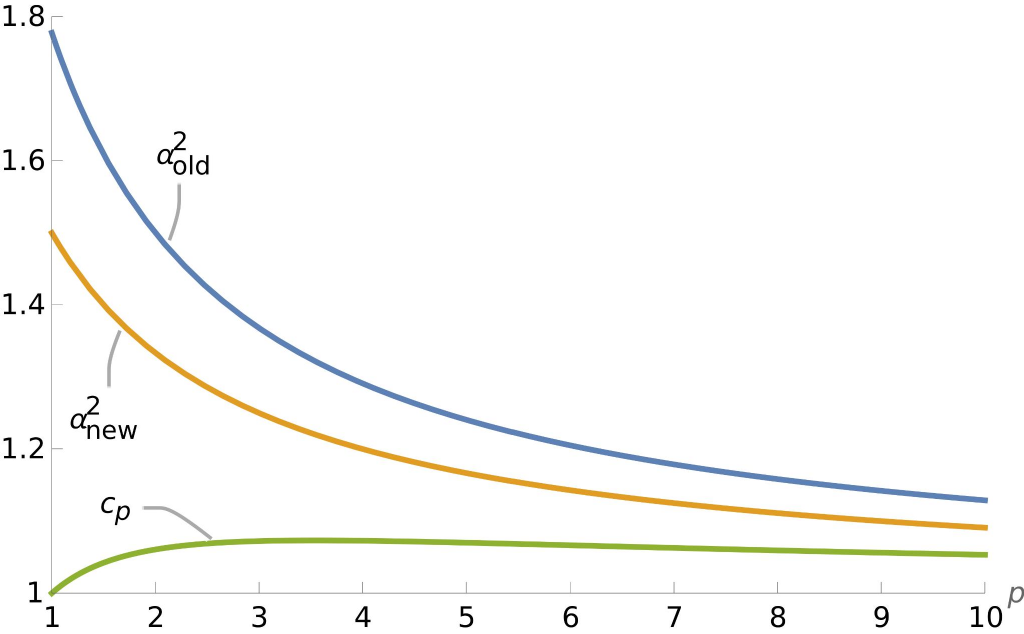}
\caption{Comparison of $\alpha_{\text{old}}^2$, $\alpha_{\text{new}}^2$ and $c_p$.}
\label{fig:alpha}
\end{figure}
The comparison between these parameters provides a clear quantitative view of the improvement gained by the optimized measure.

In \cite{NP24} we proved that the curse holds for all \(p\in(1,\infty)\) for non-negative 
quadrature formulas, the lower bounds for non-negative weights take the form
\begin{equation}\label{eq9}
N_q^{\mathrm{int},+}(\varepsilon,d)=N_p^{\mathrm{disc},+}(\varepsilon,d)
   \ge c_p^{d}(1-2\varepsilon),
   \qquad \varepsilon\in(0,1/2) . 
\end{equation}
We present the lower and the upper bounds together in one figure (Figure~\ref{fig:alpha}), even if we do not really compare the same quantities: For the lower bound we fix $\eps$, while in the upper bound we compare the asymptotic (in $N$) error.

\begin{remark}\rm
It would be interesting to derive new explicit upper bounds for 
even values of \(p\) by employing the optimized change of measure.
The computation is likely to be tedious because no closed-form expression 
is available for the optimal probability density when \(p\) differs from~1 or~2.
Nevertheless, such an approach would improve the main term
\[
\alpha_{\text{old}}(p)  = \left(\frac{2p+2}{p+2}\right)^{1/p},
\]
by replacing it with the smaller factor
\[
\alpha_{\text{new}}(p) = \left(\frac{p+2}{p+1}\right)^{1/2}.
\]
This follows from the asymptotic bound in Theorem~\ref{thm:opt:meas}.
\end{remark}

\begin{remark}\rm
For general finite $p \not=2$ we can only show asymptotic results. 
One might speculate that the asymptotic statement
\begin{equation*}
\nav_p(N,d,\varrho_d)
   \sim C_p\,\alpha_p^d\,N^{-1/2}
   \quad\text{as }N\to\infty,
\end{equation*}
(where the expectation is taken with respect to the density~\(\varrho\))
actually implies the non-asymptotic bound
\begin{equation*}
\nav_p(N,d,\varrho_d)
   \le C_p\,\alpha_p^d\,N^{-1/2}
   \qquad\text{for all }N,d.
\end{equation*}
This property is confirmed for \(p=2\) and \(\varrho=1\) (see \eqref{eq3} and \eqref{eq4}), 
but remains open for general \(p\).
\end{remark}

\begin{remark}\rm
For the periodic \(L_2\)-discrepancy (where the unit cube is considered as 
a torus and the test boxes are wrapped around), the situation differs markedly, and more is known.
The required number of points, i.e.,  the inverse periodic \(L_2\)-discrepancy, 
grows at least like \((3/2)^d\), and this lower bound holds both for QMC algorithms and 
for arbitrary positive quadrature formulas (see \cite[Theorem~2]{DHP2020}). 
Moreover, this bound is sharp, since a constructive upper bound using 
Korobov’s \(p\)-sets is possible (\cite[Remark~4]{DHP2020}). In the periodic case, a change of measure is neither required nor meaningful, 
because the space is translation invariant.

For general \(p\in[1,\infty)\) one can show using the technique of~\cite{HW12} that
\[
\nav_p^{\mathrm{per}}(N,d)
   \, \lesssim \,  \frac{\sqrt{2}}{\pi^{1/(2p)}} 
        \Gamma\!\left(\frac{p+1}{2}\right)^{1/p}
        \Bigl(\tfrac{2p+2}{p+2}\Bigr)^{d/p}N^{-1/2},
   \quad N\to\infty.
\]
This is the same asymptotic form as for the classical \( L_p^{\ast}\)-discrepancy~\eqref{uniform}.
Again, the curse of dimensionality (for random points) is most severe for small \(p\ge1\).
\end{remark}

\medskip

For our main result we consider again 
random i.i.d.\ points 
\(\bst_1,\ldots,\bst_N\in[0,1]^d\)
distributed according to a density \(\varrho_d\) and study
\[
A_{d,N}(f)=\frac{1}{N}\sum_{k=1}^N\frac{1}{\varrho_d(\bst_k)}\,f(\bst_k).
\]

The density \(\varrho_d\) depends on \(p\), and we 
restrict ourselves to product measures of the form
\(\varrho_d=\varrho^{\otimes d}:=\bigotimes_{j=1}^d\varrho\), with a one-dimensional density \(\varrho\) on \([0,1]\).
We follow the approach of \cite{HW12,St2010}.

\begin{thm}\label{thm:opt:meas}
For every \(p\in[1,\infty)\) and the optimal product density \(\varrho^{\ast,\otimes d}\) we have
\[
\lim_{N\to\infty} N^{1/2}\,\nav_p(N,d,\varrho_d^{\ast})
   \le
   \frac{\sqrt{2}}{\pi^{1/(2p)}}\,
   \Gamma\!\Bigl(\tfrac{p+1}{2}\Bigr)^{1/p}
   \Bigl(\tfrac{p+2}{p+1}\Bigr)^{d/2}.
\]
The one-dimensional optimal density is implicitly given as the solution of 
\begin{equation}\label{eq:rho}
t=\left(1-\varrho(t)\frac{p}{p+1}\right)^{2/p}\left(1+\varrho(t) \frac{2}{p+1}\right)
\end{equation}
for $t \in [0,1]$. For $p=1$ we have $$\varrho^{\ast}(t)=1+2 \cos \left(\frac{1}{3} 
\arccos(2t-1)+\frac{4 \pi}{3}\right)$$ and for $p=2$ we have 
$$\varrho^{\ast}(t)=\frac{3}{2}\sqrt{1-t}.$$
\end{thm}

\begin{proof}
We follow the proof of \cite[Theorem 2]{HW12}, 
introducing positive integration weights \(a_k=1/(N\varrho_d(\bst_k))\)
with product density \(\varrho_d\).
We then study
\begin{align*}
\mathrm{av}_p(N,d,\varrho_d)^p
 &= \!\!\int_{[0,1]^{Nd}}\!
   \int_{[0,1]^d}
   \Biggl|
     \lambda_d([\bszero,\bsx))
       -\frac{1}{N}\sum_{k=1}^N
        \frac{1}{\varrho_d(\bst_k)}
        \mathbf{1}_{[\bszero,\bsx)}(\bst_k)
   \Biggr|^p
   \rd\bsx
   \prod_{j=1}^N\varrho_d(\bst_j)\rd\bst_j,
\end{align*}
and interchange the order of integration.

For fixed \(\bsx\) define the random variables  
\(X_i=\mathbf{1}_{[\bszero,\bsx)}(\bst_i)\varrho_d(\bst_i)^{-1}\).
Their mean and variance are
\[
\E[X_i]=\lambda_d([\bszero,\bsx))=:\lambda,
\qquad
V(X_i)
   = \int_{[0,1]^d}\!
       \mathbf{1}_{[\bszero,\bsx)}(\bst)
       \tfrac{1}{\varrho_d(\bst)}\,\rd\bst
     - \lambda^2.
\]
Let
\[
X_{N,\lambda}
   = \Bigl(\frac{1}{N}\sum_{i=1}^NX_i - \lambda\Bigr)\sqrt{N}.
\]
By the central limit theorem,
\[
X_{N,\lambda}\stackrel{\mathcal{D}}{\longrightarrow}
\Bigl(\int_{[0,1]^d}\!\!
   \mathbf{1}_{[\bszero,\bsx)}(\bst)
   \tfrac{1}{\varrho_d(\bst)}\,\rd\bst
   - \lambda^2\Bigr)^{1/2}Y,
\quad Y\sim N(0,1).
\]
Proceeding as in \cite{HW12}, we obtain
\begin{align}\label{av:pro_opt}
\lim_{N\to\infty} N^{p/2}\mathrm{av}_p(N,d,\varrho_d)^p
 &= \frac{2^{p/2}}{\pi^{1/2}}
    \Gamma\!\Bigl(\tfrac{p+1}{2}\Bigr)
    \!\!\int_{[0,1]^d}\!
    \Biggl(
      \int_{[0,1]^d}
      \mathbf{1}_{[\bszero,\bsx)}(\bst)
      \tfrac{1}{\varrho_d(\bst)}\,\rd\bst
      - \lambda^2
    \Biggr)^{p/2}\!\!\!\!\rd\bsx\nonumber\\
 &\le \frac{2^{p/2}}{\pi^{1/2}}
    \Gamma\!\Bigl(\tfrac{p+1}{2}\Bigr)
    \!\!\int_{[0,1]^d}\!
    \Biggl(
      \int_{[0,1]^d}
      \mathbf{1}_{[\bszero,\bsx)}(\bst)
      \tfrac{1}{\varrho_d(\bst)}\,\rd\bst
    \Biggr)^{p/2}\!\!\!\!\rd\bsx.
\end{align}
Because of the product structure of \(\varrho_d\),
\begin{eqnarray*}
\int_{[0,1]^d}\left(\int_{[0,1]^d} \mathbf{1}_{[\bszero,\bsx)}(\bst)\tfrac{1}{\varrho_d(\bst)}\rd\bst \right)^{p/2}\rd\bsx & = & 
\int_{[0,1]^d}\left(\int_{[0,1]^d} \prod_{j=1}^d \mathbf{1}_{[0,x_j)}(t_j)\tfrac{1}{\varrho(t_j)}\rd\bst \right)^{p/2}\rd\bsx\\
& = & \int_{[0,1]^d}\prod_{j=1}^d \left(\int_0^{x_j}\tfrac{1}{\varrho(t_j)}\rd t_j \right)^{p/2}\rd\bsx\\
& = & \prod_{j=1}^d \int_0^1 \left(\int_0^{x_j}\tfrac{1}{\varrho(t_j)}\rd t_j \right)^{p/2}\rd x_j\\
& = & \left(\int_0^1 \left(\int_0^x\frac{1}{\varrho(t)}\,\rd t\right)^{p/2}\rd x\right)^d.
\end{eqnarray*}

Hence we must minimize, for given \(p\), the functional
\[
J(\varrho)
   = \int_0^1 \left(\int_0^x\frac{1}{\varrho(t)}\,\rd t\right)^{p/2}\rd x
\]
subject to \(\varrho\ge0\) and \(\int_0^1\varrho(t)\,\rd t=1\).
Define \(S(x)=\int_0^x\varrho(t)^{-1}\rd t\), 
so that \(S'(x)=1/\varrho(x)\ge0\).

Introducing a Lagrange multiplier \(\lambda\),
we minimize
\[
\mathcal{L}(S)
   = \int_0^1\!
       \bigl(S(x)^{p/2}+\lambda\varrho(x)\bigr)\rd x
   = \int_0^1\!
       \Bigl(S^{p/2}+\frac{\lambda}{S'}\Bigr)\rd x.
\]
Since the integrand does not depend explicitly on \(x\),
the Euler-Lagrange equation reduces to the Beltrami identity
\[
S^{p/2}+\frac{2\lambda}{S'}=\mu,
\]
with constant \(\mu\) (\(\mu>S^{p/2}\) on \([0,1]\)).
Integrating gives
\begin{equation}\label{equ:Sx}
2\lambda x
   = \mu S(x) - \frac{2}{p+2}S(x)^{p/2+1}.
\end{equation}
Setting \(S_1=S(1)\) yields
\begin{equation}\label{idS11}
2\lambda
   = \mu S_1 - \frac{2}{p+2}S_1^{p/2+1}.
\end{equation}

The normalization constraint \(\int_0^1\varrho(t)\,\rd t=1\) becomes
\[
1
   = \int_0^{S_1}
       \Bigl(\frac{\mu-v^{p/2}}{2\lambda}\Bigr)^2\!\rd v
   = \frac{1}{(2\lambda)^2}
       \biggl(
         \mu^2S_1
         - \frac{4\mu}{p+2}S_1^{p/2+1}
         + \frac{S_1^{p+1}}{p+1}
       \biggr),
\]
implying
\begin{equation}\label{idS2}
(2\lambda)^2
   = \mu^2S_1
     - \frac{4\mu}{p+2}S_1^{p/2+1}
     + \frac{S_1^{p+1}}{p+1}.
\end{equation}

Changing variables in the objective gives
\[
J(\varrho)
   = \frac{1}{2\lambda}
     \Bigl(
       \frac{2\mu}{p+2}S_1^{p/2+1}
       - \frac{1}{p+1}S_1^{p+1}
     \Bigr).
\]
After eliminating \(\lambda\) via \eqref{idS11}–\eqref{idS2} 
and some algebra, one obtains
\begin{equation}\label{def:mu:lam}
\mu
   = \frac{S_1^{p/2}}{p+2}
       \Bigl(
         2 + \frac{p}{\sqrt{p+1}\sqrt{S_1-1}}
       \Bigr),
\qquad
2 \lambda
   = \frac{p}{(p+2)\sqrt{p+1}}
       \frac{S_1^{p/2+1}}{\sqrt{S_1-1}}.
\end{equation}
Substituting these into \(J(\varrho)\) gives
\begin{equation}\label{idJ2}
J(\varrho)
   = \frac{S_1^{p/2}}{p+2}
       \Bigl(
         2 - \frac{p}{\sqrt{p+1}}\sqrt{S_1-1}
       \Bigr).
\end{equation}
Minimization of \eqref{idJ2} yields the optimal
\begin{equation}\label{S1opt}
S_1^{\ast}
   = \frac{p+2}{p+1},
\qquad
J_{\min}
   = \frac{1}{p+1}
       \Bigl(\frac{p+2}{p+1}\Bigr)^{p/2}.
\end{equation}
Substituting \(J_{\min}\) into~\eqref{av:pro_opt}, taking the \(p\)th root,
and dividing by the initial error \((p+1)^{-d/p}\) produces the claimed bound.
\medskip

It remains to compute the optimal one-dimensional densities $\varrho^{\ast}$. 
Inserting $S_1=S_1^{\ast}$ from \eqref{S1opt} into \eqref{def:mu:lam} yields $\mu=(S_1^{\ast})^{p/2}$ 
and $2 \lambda= \frac{p}{p+1} (S_1^{\ast})^{p/2}$. Inserting into \eqref{equ:Sx} yields the equation 
\begin{equation}\label{eq:Sx:p1}
\frac{p}{p+1} (S_1^{\ast})^{p/2} x= (S_1^{\ast})^{p/2} S(x)-\frac{2}{p+2} S(x)^{p/2+1}.
\end{equation}
Differentiating with respect to $x$ yields $$\frac{1}{\varrho(x)}=S'(x)=\frac{\frac{p}{p+1} 
(S_1^{\ast})^{p/2}}{(S_1^{\ast})^{p/2}-S(x)^{p/2}}.$$ Thus, $$S(x)=S_1^{\ast} \left(1-\varrho(x) \frac{p}{p+1}\right)^{2/p}.$$ 
Re-inserting into \eqref{eq:Sx:p1} and using $S_1^{\ast}=\frac{p+2}{p+1}$ leads, 
after some lines of tedious algebra, to the equation \eqref{eq:rho}.

For $p=1$ equation \eqref{eq:rho} reduces to $$4 t=(2-\varrho(t))^2(1+\varrho(t)),$$ which we solve 
for $\varrho=\varrho(t)$, $t \in [0,1]$. Expanding gives the monic form $$\varrho^3-3 \varrho^2+4(1-t)=0.$$ 
With the substitution $\varrho=y+1$ we obtain the cubic equation $$y^3-3y+(2-4t)=0,$$ which is of the form 
$y^3+py+q=0$ with $p=-3$ and $q=2-4t$. For $t \in (0,1)$ the discriminant is positive and all three roots 
are real (casus irreducibilis) which are according to the trigonometric form of Cardano's solution given 
as $$y=2 \cos\left(\frac{1}{3} \arccos(2t-1)+\frac{2 \pi k}{3}\right)\quad \mbox{for $k\in \{0,1,2\}$.}$$ 
Choosing the branch $k=2$ in order to satisfy $\varrho(0)=2$ and $\varrho(1)=0$ yields 
$$\varrho^{\ast}(t)=y+1=1+2 \cos \left(\frac{1}{3} \arccos(2t-1)+\frac{4 \pi}{3}\right),$$ as claimed. 

For $p=2$ equation \eqref{eq:rho} reduces to $t=1-\frac{4}{9} \varrho(t)^2$, 
whose solution is easily seen to be $$\varrho^{\ast}(t)=\frac{3}{2}\sqrt{1-t},$$ as claimed.
\end{proof}

\begin{remark}\rm
Explicit closed forms of the optimal density $\varrho^{\ast}$ are available only 
for \(p=1\) and \(p=2\). The graphs are shown in Figure~\ref{f1}.
\begin{figure}
\centering
\includegraphics[width=0.6 \textwidth]{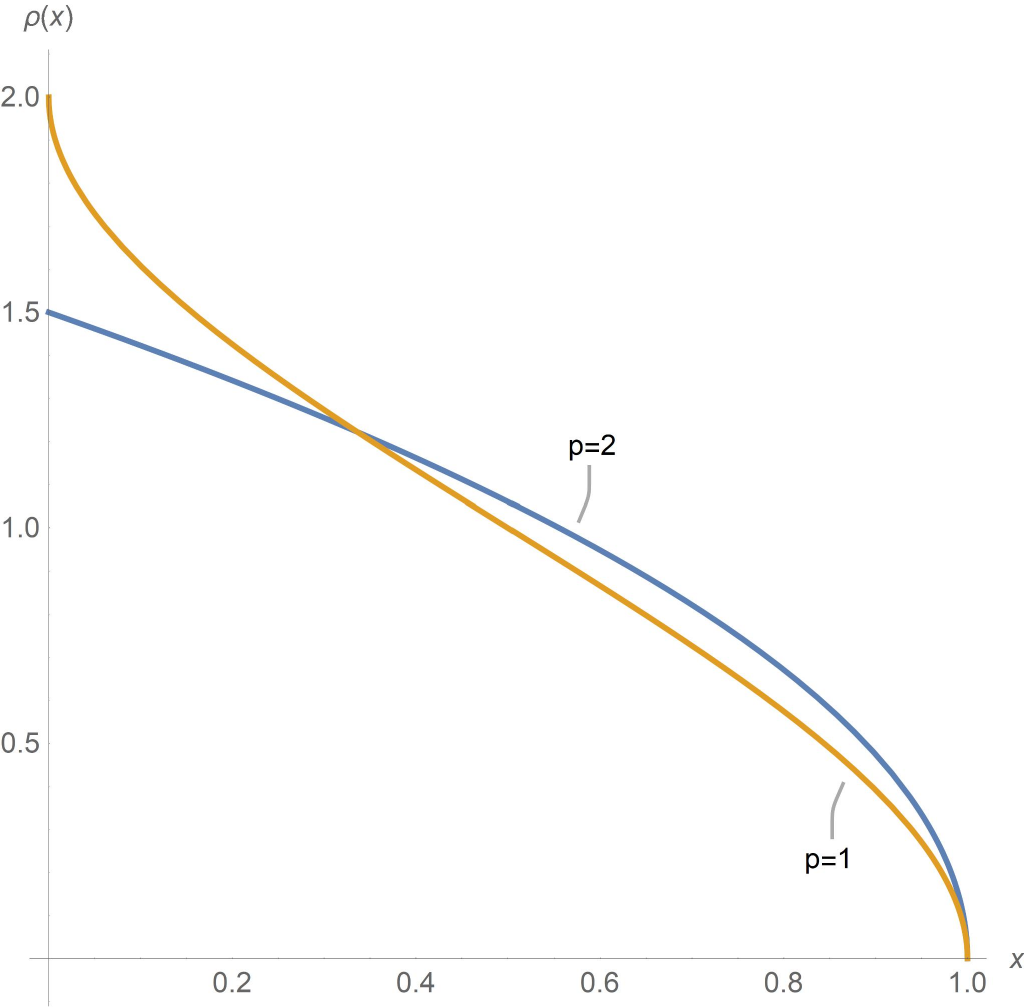}
\caption{Optimal density $\varrho^*$ for $p=1$ and $p=2$.}
\label{f1}
\end{figure}
Nevertheless, the density function \(\varrho^{\ast}\) can be computed numerically 
for any given \(p\) by solving equation \eqref{eq:rho}.
For illustration, Figure~\ref{f2} shows the numerically computed optimal density for \(p=10\) and \(p=100\).
\begin{figure}
\centering
\includegraphics[width=0.6 \textwidth]{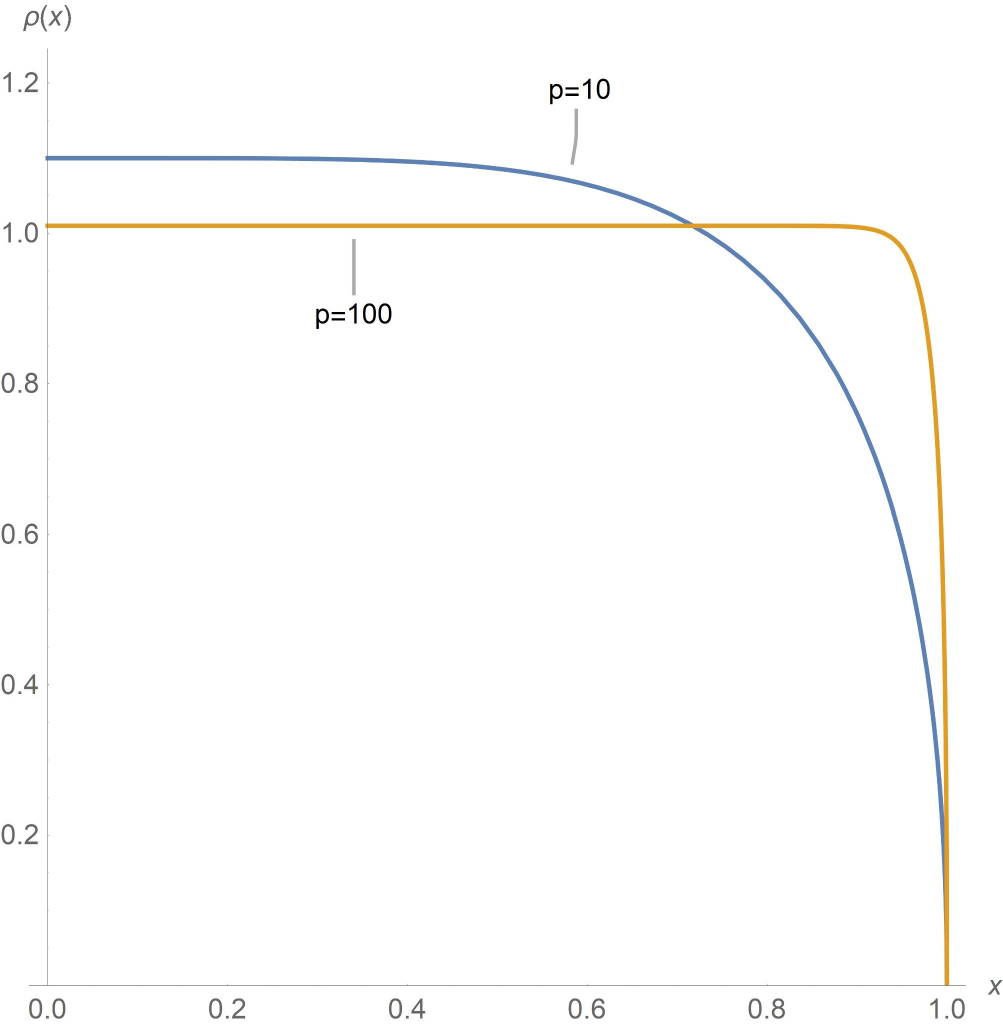}
\caption{Optimal densities $\varrho^*$ for $p=10$ and $p=100$.}
\label{f2}
\end{figure}
\end{remark}

\section{Stability, and open problems}\label{stability}

We conclude with a discussion of stability, which in our context must be understood 
relative to the function space \(F_{d,q}\) rather than in the classical sup-norm sense.

QMC rules $A_{d,N}^{\ast}$ (see Section~\ref{sec:back}) integrate constant functions exactly and are traditionally regarded as stable because
\[
\| A_{d,N}^{\ast}\|_\infty = 1.
\]
Here the operator acts on \(C([0,1]^d)\), and its norm coincides with that of the integration operator itself. Considering linear  quadrature formulas of the form
\[
A_{d,N}(f)=\sum_{k=1}^N a_k f(\bst_k),
\]
a frequently used indicator of stability is
\begin{equation}\label{stab-old}
\|A_{d,N}\|_\infty=\sum_{k=1}^N|a_k|,
\end{equation}
which ideally equals~1 (or is close to~1) for a well-balanced formula.

In the present framework, however, the measure~\eqref{stab-old} 
is not appropriate because we work in \(F_{d,q}\).
Note that the integration operator on \(F_{d,q}\) has norm \((p+1)^{-d/p}\), 
which is exponentially small in \(d\).
A stable quadrature formula for \(F_{d,q}\) should therefore also have 
an exponentially small operator norm, and there is no requirement that 
it integrates  constants exactly, since constant functions are not contained in \(F_{d,q}\).

It can be verified that the norm of a QMC rule \(A_{d,N}^{\ast}\) 
in \(F_{d,q}\) depends strongly on the node set.
The inequality
\[
0\le\|A_{d,N}^{\ast}\|_{F_{d,q}}\le1
\]
is sharp: if each point \(\bst_k\) has at least one coordinate equal to~1 we obtain the lower bound,
while for \(\bst_k=\boldsymbol{0}\) for all \(k\) we obtain the upper bound.
Since \(1\) is exponentially larger than the initial error \((p+1)^{-d/p}\),
some QMC algorithms possess very large operator norms and are thus unstable on \(F_{d,q}\);
naturally, they also exhibit large integration errors.

In contrast, the algorithms considered in this paper use comparatively large 
(non-negative) weights, which may lead to arbitrarily large values of \(\|A_{d,N}^+\|_\infty\).
Consequently, these quadrature formulas are not stable in the classical sense.
Nevertheless, they achieve small integration errors, and their
\emph{\(F_{d,q}\)-stability}, measured by \(\|A_{d,N}^+\|_{F_{d,q}}\), remains small.
Thus, in our setting, algorithms \(A_{d,N}^+\) with small integration error 
\(I_d-A_{d,N}^+\) automatically satisfy small \(F_{d,q}\)-norms by the triangle inequality.

As a simple example, for \(p=2\) and \(d=1\) the optimal one-point rule is 
\[
A_{1,1}(f)=\tfrac{2}{3}f(\tfrac{1}{3}),
\]
with worst-case error \(1/\sqrt{27}\).
If the weight is fixed to \(a_1=1\), the optimal node is \(t_1=\tfrac{1}{2}\), 
and the error equals \(1/\sqrt{12}\).
With the optimal density \(\varrho^{\ast}\) one obtains
\[
\sup_{f\in F_{1,2}}\frac{f(t)}{\varrho^{\ast}(t)}=\frac{2}{3}\qquad \mbox{for } t\in[0,1),
\]
so all possible contributions \(\tfrac{f(t_k)}{N \varrho^{\ast}(t_k)}\) lie in the same interval \([-2/(3N),2/(3N)]\);
that is, each function value contributes uniformly to the result.

\begin{remark}\rm 
A practical issue arises when some of the weights \(a_k\) are large:
if the available data are noisy or the function \(f\) does not belong to \(F_{d,q}\),
the algorithm may behave poorly.
This risk can be mitigated by replacing the optimal density \(\varrho^{\ast}\) with 
a slightly regularized version that stays bounded away from zero by a positive constant.
\end{remark}

\paragraph{Open problems.}
Several questions remain open and suggest directions for further work:

\begin{enumerate}
\item Can the asymptotic bounds \eqref{uniform}, \eqref{new-density}, 
\eqref{eq7}, and \eqref{eq8}, be converted into explicit non-asymptotic bounds valid for all~\(N\)?

\item Can the upper bounds be improved when arbitrary quadrature 
formulas are permitted, i.e., when some weights \(a_k\) may be negative?

\item Can one close the gap between upper and lower bounds by 
establishing sharper lower bounds, particularly for \(p<2\)?

\item Does the curse of dimensionality hold for \(p=1\)?
\end{enumerate}

\vspace{0.5cm}
\noindent{\bf Author’s Address:}\\[1ex]
\noindent
Erich Novak, Institut f\"ur Mathematik, FSU Jena, Inselplatz 5, 07743 Jena, Germany. 
E-mail: \texttt{erich.novak@uni-jena.de}\\[1ex]
\noindent
Friedrich Pillichshammer, Institut f\"ur Finanzmathematik und Angewandte Zahlentheorie, JKU Linz, 
Altenbergerstra{\ss}e 69, A-4040 Linz, Austria. E-mail: \texttt{friedrich.pillichshammer@jku.at}

\end{document}